\documentclass[11pt,reqno]{amsart}
\usepackage{indentfirst,enumerate,cite,amssymb,amsfonts,amsmath,amsthm,mathrsfs,dsfont}
\usepackage{color}
\usepackage{graphicx}
\usepackage{multicol}
\usepackage[colorlinks=true,linkcolor=blue,citecolor=blue]{hyperref}
\usepackage{enumerate}
\usepackage{geometry}
\numberwithin{equation}{section}

\newtheorem{theorem}{Theorem}[section]
\newtheorem{definition}[theorem]{Definition}
\newtheorem{example}[theorem]{Example}

\newtheorem{proposition}[theorem]{Proposition}
\newtheorem{corollary}[theorem]{Corollary}

\newtheorem{remark}{Remark}

\allowdisplaybreaks

\begin{document}
	
	\title[Perturbations of globally hypoelliptic pseudo-differential operators on $\mathbb{R}^n$]{Perturbations of globally hypoelliptic pseudo-differential operators on $\mathbb{R}^n$}
	
	\author[P. Tokoro]{Pedro M. Tokoro}
	\address{
		Programa de P\'os-Gradua\c c\~ao em Matem\'atica,  
		Universidade Federal do Paran\'a,  
		Caixa Postal 19096\\  CEP 81530-090, Curitiba, Paran\'a,  
		Brasil}
	\email{pedro.tokoro@ufpr.br}
	\thanks{The author thanks Alexandre Kirilov and Wagner A. A. de Moraes for the useful discussions. This study was financed in part by the Coordena\c c\~ao de Aperfei\c coamento de Pessoal de N\'ivel Superior - Brasil (CAPES) - Finance Code 001.}
	
	
	
	\begin{abstract}
		This paper demonstrates the stability of the global regularity for a class of pseudo-differential operators under lower-order perturbations. We establish that if an operator has a globally hypoelliptic symbol, its global regularity (in the sense of Schwartz functions and tempered distributions) is preserved when perturbed by operators of sufficiently lower order. This result applies in particular to operators within the Shubin and SG classes. Furthermore, we discuss why this stability result does not hold in the standard H\"ormander classes.
	\end{abstract}
	
	\maketitle
	
	\section{Introduction}
	
	This paper investigates the stability of global regularity for a class of pseudo-differential operators when subjected to lower-order perturbations. We focus on operators that satisfy a condition generalizing the standard notion of loss of derivatives in Sobolev spaces associated with temperate weights.
	
	The stability of hypoellipticity under perturbations has been a significant area of research. Alberto Parmeggiani, in \cite{Par2015}, studied the local hypoellipticity of differential operators $P$ satisfying subelliptic estimates, under perturbations $A$ whose order is strictly less than the loss of derivatives of $P$ on the scale of local Sobolev spaces. These findings were later extended to more general pseudo-differential operators by Parenti and Parmeggiani in \cite{PP2018}. See also \cite{PP2005} for results on operators with loss of derivatives.
	
	In a global context, the stability of global hypoellipticity for subelliptic operators on the torus has been explored in various frameworks, including smooth, Gevrey, and Denjoy-Carleman (e.g., \cite{BRCCJ2016,CC2017,FP2023,FP2024,FPV2020}). In \cite{BRCCJ2016}, the authors also presented results on perturbations of sums of squares of vector fields on the product of a compact Lie group with a closed manifold. Our approach, given the global nature of our results, aligns more closely with the methods used for operators on the torus than with Parenti and Parmeggiani's locally-focused ideas, which rely on singular supports and local Sobolev spaces.
	
	Let $\mathcal{S}(\mathbb{R}^n)$ be the Schwartz space on $\mathbb{R}^n$ and $\mathcal{S}'(\mathbb{R}^n)$ the space of temperate distributions. For a linear operator $P:\mathcal{S}'(\mathbb{R}^n)\to \mathcal{S}'(\mathbb{R}^n)$ such that $P(\mathcal{S}(\mathbb{R}^n))\subset \mathcal{S}'(\mathbb{R}^n)$, we say that $P$ is $\mathcal{S}$-globally hypoelliptic if
	\[u\in\mathcal{S}'(\mathbb{R}^n),\ Pu\in\mathcal{S}(\mathbb{R}^n),\ \Rightarrow\ u\in\mathcal{S}(\mathbb{R}^n).\]
	
	More precisely, we will consider $P$ a pseudo-differential operator with symbol $p(x,\xi)$ satisfying estimates of the form:
	\[|\partial_\xi^\alpha\partial_x^\beta p(x,\xi)| \lesssim M(x,\xi)\Psi(x,\xi)^{-|\alpha|}\Phi(x,\xi)^{-|\beta|},\quad\forall (x,\xi)\in\mathbb{R}^{2n},\ \forall  \alpha,\beta\in\mathbb{N}_0^n.\]
	
	Here, $M$ is a temperate weight and $\Phi,\Psi$ are sub-linear temperate weights (see Definitions \ref{sublinear} and \ref{temperate}). Here, given two non-negative functions $f,g:X\to[0,+\infty)$ on a set $X$, the notation $f(x)\lesssim g(x)$ means that there exists $C>0$ such that $f(x)\leq Cg(x)$, for all $x\in X$.
	
	These symbol classes are similar to those introduced by Beals in \cite{Beals} and are particular cases of the symbol classes of the Weyl-H\"ormander calculus \cite{Hormander_Weyl} for appropriately chosen slowly varying metrics on the phase space $T^*\mathbb{R}^n = \mathbb{R}^{2n}$. A comprehensive treament of the pseudo-differential calculus in the classes considered in this work can be found in the first chapter of \cite{NicRod}.
	
	A temperate weight $M$ generalizes the notion of the order of an operator. Under certain assumptions, we can define the Sobolev space $H(M)$ associated with $M$. For an operator $P$ with symbol in $S(M';\Phi,\Psi)$ satisfying hypoelliptic estimates, we can obtain a weight $M_0\lesssim M'$ such that, for every weight $M$,
	\begin{equation}\label{loss_intro}
		u\in\mathcal{S}'(\mathbb{R}^n),\ Pu\in H(M)\ \Rightarrow\ u\in H(MM_0),
	\end{equation}
	which generalizes the usual notion of hypoellipticity with loss of derivatives:
	\[u\in\mathcal{D}'(\Omega),\ Pu\in H^s(\Omega)\ \Rightarrow \ u\in H^{s+m'}(\Omega),\]
	where $\Omega$ is an open subset of $\mathbb{R}^n$ or a smooth manifold, $H^s(\Omega)$, $s\in\mathbb{R}^n$, are the usual Sobolev spaces on $\Omega$, $P$ is a pseudo-differential operator of order $m$ on $\Omega$ and $m'\leq m$. 
	
	The primary objective of this work is to demonstrate that if $P$ satisfies condition \eqref{loss_intro} and $A\in S(\tilde M;\Phi,\Psi)$, where $\tilde M$ has strictly smaller growth order than $M_0$, then the perturbed operator $P+A$ still satisfies \eqref{loss_intro}. This implies that $P+A$ remains $\mathcal{S}$-globally hypoelliptic. We also examine the specific case of operators with symbols in Shubin and SG classes, which can be derived by selecting appropriate $M,\Phi,\Psi$. Furthermore, we show that our results do not extend to the standard H\"ormander classes $S^m_{\rho,\delta}(\mathbb{R}^n)$. Finally, we stablish some connections between a weaker notion of $\mathcal{S}$-global hypoellipticity and the global solvability of $P$.

	\section{Global pseudo-differential operators}
	
	In this section, we lay out the key definitions and essential facts required for our results. For detailed proofs and further discussion, see \cite[Chapter 1]{NicRod}.
	
	We say that a continuous function $\Phi:\mathbb{R}^{2n}\to (0,+\infty)$ is a sub-linear weight if
	\begin{equation}\label{sublinear}
		1 \leq \Phi(x,\xi) \lesssim 1+|x|+|\xi|, \quad \forall (x,\xi)\in\mathbb{R}^{2n},
	\end{equation}
	and a temperate weight if there exists $s>0$ such that
	\begin{equation}\label{temperate}
		\Phi(x+y,\xi+\eta) \lesssim \Phi(x,\xi)(1+|y|+|\eta|)^s,\quad\forall (x,\xi),(y,\eta)\in\mathbb{R}^{2n}.
	\end{equation}
	
	If $\Phi$ is a temperate weight, if we put $x'+y$ in the place of $x$, $\xi'+\eta$ in the place of $\xi$, $-y$ in the place of $y$, and $-\eta$ in the place of $\eta$, estimate \eqref{temperate} yields the lower bound
	\begin{equation*}
		\Phi(x'+y,\xi'+\eta)\gtrsim \Phi(x',\xi')(1+|y|+|\eta|)^{-s},\quad\forall (x',\xi'),(y,\eta)\in\mathbb{R}^{2n}.
	\end{equation*}
	
	In particular, taking $x'=\xi'=0$, we obtain
	\begin{equation*}
		(1+|y|+|\eta|)^{-s}\lesssim \Phi(y,\eta)\lesssim (1+|y|+|\eta|)^s,\quad\forall (y,\eta)\in\mathbb{R}^{2n}.
	\end{equation*}
	
	Fix two temperate sub-linear weights $\Phi,\Psi$. Given a temperate weight $M$, we denote by $S(M;\Phi,\Psi)$ the space of all $p(x,\xi)\in C^\infty(\mathbb{R}^{2n})$ such that, for every $\alpha,\beta\in\mathbb{N}_0^n$ and $(x,\xi)\in\mathbb{R}^{2n}$, we have
	\begin{equation*}
		|\partial_\xi^\alpha\partial_x^\beta p(x,\xi)|\lesssim M(x,\xi)\Psi(x,\xi)^{-|\alpha|}\Phi(x,\xi)^{-|\beta|}.
	\end{equation*}
	
	Any symbol $p(x,\xi)\in S(M;\Phi,\Psi)$ defines an operator $p(x,D)$ on $\mathcal{S}(\mathbb{R}^n)\to \mathcal{S}(\mathbb{R}^n)$ via the usual left-quantization
	\begin{equation*}
		\mathrm{Op}(p)u(x) = p(x,D)u(x) = (2\pi)^{-\frac{n}{2}}\int_{\mathbb{R}^n} e^{i\xi\cdot x}p(x,\xi)\widehat{u}(\xi)\,\mathrm{d}\xi,\quad u\in\mathcal{S}(\mathbb{R}^n),
	\end{equation*}
	which can be extended to a continuous operator $\mathcal{S}'(\mathbb{R}^n)\to \mathcal{S}'(\mathbb{R}^n)$. In this case, we say that $p(x,D)$ is a pseudo-differential operator with symbol $p(x,\xi)$.
	
	We denote by $OPS(M;\Phi,\Psi)$ the space of pseudo-differential operators on $\mathbb{R}^n$ with symbol in $S(M;\Phi,\Psi)$.
	
	Given two temperate sub-linear weights $\Phi,\Psi$, the function
	\[h(x,\xi)=\Phi(x,\xi)^{-1}\Psi(x,\xi)^{-1},\quad (x,\xi)\in\mathbb{R}^{2n},\]
	will be called the Planck function.

	We may also assume there exists $\delta>0$ such that
	\begin{equation*}
		\Phi(x,\xi)\Psi(x,\xi)\gtrsim (1+|x|+|\xi|)^\delta,\quad\forall (x,\xi)\in\mathbb{R}^{2n}.
	\end{equation*}
	
	This property is called the strong uncertainty principle. An important consequence of the strong uncertainty principle is the following:
	\begin{equation*}
		\bigcap_{j=0}^\infty S(Mh^j;\Phi,\Psi) = \mathcal{S}(\mathbb{R}^{2n}).
	\end{equation*}
	
	If $r(x,\xi)\in\mathcal{S}(\mathbb{R}^{2n})$, then $r(x,D)$ extends to an operator $\mathcal{S}'(\mathbb{R}^n)\to \mathcal{S}(\mathbb{R}^n)$. In this case, we say that $r(x,D)$ is a regularizing operator.
	
	We say that a symbol $p\in S(M;\Phi,\Psi)$ is globally elliptic if there exists $R>0$ such that
	\begin{equation*}
		|p(x,\xi)|\gtrsim M(x,\xi),\quad |x|+|\xi|\geq R.
	\end{equation*}
	
	More generally, we say that a symbol $p(x,\xi)\in S(M;\Phi,\Psi)$ is globally hypoelliptic if there exists a temperate weight $M_0\lesssim M$ and $R >0$ such that
	\begin{equation*}
		|p(x,\xi)|\gtrsim M_0(x,\xi),\quad |x|+|\xi|\geq R,
	\end{equation*}
	and, for every $\alpha,\beta\in\mathbb{N}_0^n$, we have
	\begin{equation*}
		|\partial_\xi^\alpha\partial_x^\beta p(x,\xi)|\lesssim |p(x,\xi)|\Psi(x,\xi)^{-|\alpha|}\Phi(x,\xi)^{-|\beta|},\quad |x|+|\xi|\geq R.
	\end{equation*}
	
	We denote by $\mathrm{Hypo}(M,M_0;\Phi,\Psi)$ the set of such symbols.
	
	Clearly, every globally elliptic symbol is also globally hypoelliptic with $M_0=M$. The second condition on the definition of globally hypoelliptic symbols is satisfied since we have
	\[p(x,\xi)\lesssim M(x,\xi)\lesssim p(x,\xi),\]
	for big enough $|x|+|\xi|$.
	
	\begin{theorem}\label{parametrix}
		Assume the strong uncertainty principle and suppose that $P$ has symbol in $\mathrm{Hypo}(M,M_0;\Phi,\Psi)$. Then, there exists an operator $Q\in OPS(M_0^{-1};\Phi,\Psi)$ and regularizing operators $S_1,S_2$ such that
		\begin{equation*}
			QP=I+S_1\quad\text{and}\quad PQ=I+S_2, 
		\end{equation*}
		where $I$ is the identity operator. In this case, we say that $Q$ is a parametrix of $P$.
	\end{theorem}
	
	\begin{remark}
		The parametrix is unique modulo regularizing operators.
	\end{remark}
	
	\begin{definition}
		We say that an operator $P:\mathcal{S}(\mathbb{R}^n)\to \mathcal{S}(\mathbb{R}^n)$, $\mathcal{S}'(\mathbb{R}^n)\to \mathcal{S}'(\mathbb{R}^n)$ is $\mathcal{S}$-globally hypoelliptic if
		\begin{equation*}
			u\in \mathcal{S}'(\mathbb{R}^n),\ Pu\in \mathcal{S}(\mathbb{R}^n)\ \Rightarrow\ u\in \mathcal{S}(\mathbb{R}^n).
		\end{equation*}
	\end{definition}
	
	As a consequence of the existence of parametrices, we have the following:
	
	\begin{corollary}
		Assume the strong uncertainty principle. If the operator $P$ has symbol in $\mathrm{Hypo}(M,M_0;\Phi,\Psi)$, then $P$ is $\mathcal{S}$-globally hypoelliptic.
	\end{corollary}
	
	We say that a temperate weight $M$ is a regular weight if $M$ is a symbol in its own class, that is, $M\in S(M;\Phi,\Psi)$. The existence of regular weights is related to the existence of elliptic symbols in a given class. Under slow variation assumptions on the sub-linear weights $\Phi,\Psi$, we can always obtain regular weights, see \cite[Section 1.3]{NicRod} for a detailed discussion.
	
	\begin{definition}
		Fix $\Phi,\Psi$, let $M$ be a regular weight, and consider $P$ a pseudo-differential operator with elliptic symbol $p(x,\xi)\in S(M;\Phi,\Psi)$. Let $Q$ be a left-parametrix of $P$, that is, $QP+I=S$, where $S$ is a regularizing operator. Then, we define the Sobolev space
		\begin{equation*}
			H(M)=\{u\in \mathcal{S}'(\mathbb{R}^n)\,:\, Pu\in L^2(\mathbb{R}^n)\},
		\end{equation*}
		which is a Banach space endowed with the norm
		\begin{equation*}
			\|u\|_{H(M)}=\|Pu\|_{L^2}+\|Su\|_{L^2}.
		\end{equation*}
	\end{definition}
	
	Actually, $H(M)$ is a Hilbert space with the inner product
	\[(u,v)_{H(M)} = (Pu,Pv)_{L^2}+(Su,Sv)_{L^2},\quad u,v\in H(M),\]
	which induces a norm equivalent to $\|\cdot\|_{H(M)}$, see \cite[Proposition 1.5.3]{NicRod}.
	
	The term $\|Su\|_{L^2}$ is necessary in order to obtaining a norm. One can show that this definition is independent on the operator $P$ in the sense that different choices of elliptic operators yields equivalent norms. Moreover, using Anti-Wick techniques, one can show also that these spaces do not depend on $\Phi,\Psi$, as shown in \cite[Section 1.7]{NicRod}. For any regular weight $M$, we have continuous inclusions
	\begin{equation*}
		\mathcal{S}(\mathbb{R}^n) \hookrightarrow H(M) \hookrightarrow \mathcal{S}'(\mathbb{R}^n),
	\end{equation*}
	where the inclusion of $\mathcal{S}(\mathbb{R}^n)$ in $H(M)$ is also dense.
	
	We say that a weight $M$ tends to zero at infinity if
	\[M(x,\xi) \to 0,\quad |x|+|\xi|\to\infty.\]
	
	\begin{proposition}\label{cont_sobolev}
		Let $P$ be an operator with symbol in $S(M;\Phi,\Psi)$, where $M$ is a regular weight.
		\begin{enumerate}
			\item For every regular weight $M'$, $P$ defines a bounded operator $H(M')\to H(M'/M)$.
			\item If $M_1,M_2$ are regular weights such that $MM_2/M_1$ tends to zero at infinity, then $P$ defines a compact operator $H(M_1)\to H(M_2)$.
		\end{enumerate}
	\end{proposition}
	
	In particular, if $M_2\lesssim M_1$, then the inclusion $H(M_1)\hookrightarrow H(M_2)$ is continuous and, if $M_2/M_1$ tends to zero at infinity, then the inclusion $H(M_1)\hookrightarrow H(M_2)$ is also compact. Moreover, if $\Phi$ and $\Psi$ are regular weights and the strong uncertainty principle holds, we have the topological equalities
	\[\bigcap_{s,t}H(\Phi^t\Psi^s)=\mathcal{S}(\mathbb{R}^n)\quad\text{and}\quad \bigcup_{s,t}H(\Phi^t\Psi^s)=\mathcal{S}'(\mathbb{R}^n).\]
	
	In particular, $\mathcal{S}(\mathbb{R}^n)$ is a FS space and $\mathcal{S}'(\mathbb{R}^n)$ is a DFS space, see \cite{Komatsu1967}.
	
	Still assuming the strong uncertainty principle, the existence of a parametrix for operators with globally hypoelliptic symbols implies the following estimate on Sobolev spaces:
	
	\begin{proposition}
		Suppose that $P$ has symbol in $\mathrm{Hypo}(M',M_0;\Phi,\Psi)$ for certain temperate weights $M',M_0$, $M_0\lesssim M'$. Then, for every temperate weights $M,\tilde M$, we have
		\[\|u\|_{H(MM_0)}\lesssim \|Pu\|_{H(M)}+\|u\|_{H(\tilde M)},\quad u\in\mathcal{S}(\mathbb{R}^n).\]
		
		In particular, for every temperate weight $M$, we have
		\begin{equation}\label{gh_loss}
			u\in\mathcal{S}'(\mathbb{R}^n),\ Pu\in H(M)\ \Rightarrow\ u\in H(MM_0).
		\end{equation}
	\end{proposition}
	
	Here, \eqref{gh_loss} generalizes the notion of hypoellipticity with loss of derivatives. Assuming that $\Phi$ and $\Psi$ are regular weights satisfying strong uncertainty principle, we have that condition \eqref{gh_loss} implies the $\mathcal{S}$-global hypoellipticity of $P$. Indeed, suppose that $u\in\mathcal{S}'(\mathbb{R}^n)$ is such that
	\[Pu\in \mathcal{S}(\mathbb{R}^n) = \bigcap_{N\in\mathbb{N}}H(h^N).\]
	
	Then, for each $N\in\mathbb{N}$, we have $u\in H(h^NM_0)$, which gives us
	\[u\in \bigcap_{N\in\mathbb{N}}H(h^NM_0)=\mathcal{S}(\mathbb{R}^n).\]
	
	\section{The stability theorem}
	
	Now, we can prove our main theorem. Fix two sub-linear regular weights $\Phi,\Psi$ and assume the strong uncertainty principle.
	
	\begin{theorem}\label{thm_gh}
		 Let $M'$ be a regular weight and $P\in OPS(M';\Phi,\Psi)$ be a pseudo-differential operator satisfying \eqref{gh_loss}, for every regular weight $M$ and some fixed $M_0$. Suppose that $A\in OPS(\tilde{M};\Phi,\Psi)$, where $\tilde M$ is a regular weight such that $M_0\tilde{M}^{-1}\gtrsim h^{-\varepsilon}$ for some $\varepsilon>0$, where $h=\Phi^{-1}\Psi^{-1}$ is the Planck function. Then $P+A$ also satisfies \eqref{gh_loss} for every $M$ and the same $M_0$. In particular, $P+A$ is $\mathcal{S}$-globally hypoelliptic.
	\end{theorem}
	\begin{proof}
		Write $L=P+A$ and let $u\in\mathcal{S}'(\mathbb{R}^n)$ be such that $Lu\in H(M)$, for some regular weight $M$. Let us show that $u\in H(MM_0)$.
		
		First, since $u\in\mathcal{S}'(\mathbb{R}^n)$ and the strong uncertainty principle holds, there exists a regular weight $M_u$ such that $u\in H(M_u)$. By Proposition \ref{cont_sobolev}, we have that $Au\in H(M_u/\tilde M)$. Hence,
		\[Pu=Lu-Au\in H(M_1),\]
		where $M_1\in\{M,M_u/\tilde M\}$ is such that $M_1\lesssim M$ and $M_1\lesssim M_u/\tilde M$.
		
		If $M_1=M$, then $Pu\in H(M)$, which implies $u\in H(MM_0)$ by hypothesis. If $M_1=M_u/\tilde M$, then $Pu\in H(M_u/\tilde M)$ implies $u\in H(M_uM_0/\tilde M)$. Hence, $Au\in H(M_uM_0/\tilde{M}^2)$. In this case, we have
		\[Pu=Lu-Au\in H(M_2),\]
		where $M_2\in\{M,M_uM_0/\tilde{M}^2\}$ is such that $M_2\lesssim M$ and $M_2\lesssim M_uM_0/\tilde{M}^2$. If $M_2=M$, then $u\in H(MM_0)$ as before. Otherwise, $u\in H(M_uM_0^2/\tilde{M}^2)$.
		
		Proceeding recursively, after finitely many times, we may obtain $u\in H(MM_0)$. Indeed, by the strong uncertainty principle, we have
		\[M_u\frac{M_0^N}{\tilde{M}^N}\gtrsim M_uh^{-N\varepsilon} \gtrsim MM_0,\]
		for sufficiently big $N\in\mathbb{N}$, which concludes the proof.		
	\end{proof}
	
	\begin{remark}
		The previous result still holds if we consider systems of pseudo-differential operators satisfying \eqref{gh_loss}, as in \cite[Theorem 3.4]{FP2023}.
	\end{remark}

	\section{Applications}
	
	
	This section introduces two important classes of symbols, namely, the Shubin and SG classes, to which Theorem \ref{thm_gh}, is applicable. We will also briefly discuss why the perturbation property does not hold for operators whose symbols belong to the standard H\"ormander classes $S_{\rho,\delta}^m(\mathbb{R}^n)$. For a detailed exposition on the Shubin and SG classes, we refer the reader to Chapters 2 and 3 of \cite{NicRod}. A complete treatment of the pseudo-differential calculus in the H\"ormander classes can be found in \cite{Kumano-go}.
	
	\subsection{Shubin classes}
	
	Given $m\in\mathbb{R}$, we define the Shubin classes of symbols as
	\[\Gamma^m(\mathbb{R}^n) = S(\langle z\rangle^m;\langle z\rangle,\langle z\rangle),\]
	where $z=(x,\xi)\in\mathbb{R}^{2n}$ and $\langle z\rangle = (1+|x|^2+|\xi|^2)^{\frac{1}{2}}$. Notice that $\langle z\rangle$ is a regular sub-linear weight and the strong uncertainty principle holds. We denote by $OP\Gamma^m(\mathbb{R}^n)$ the space of pseudo-differential operators with symbol in $\Gamma^m(\mathbb{R}^n)$. We have
	\[\mathrm{Op}(\langle z\rangle^2) = 1-\Delta+|x|^2 = 1+\mathcal{H},\]
	where $\mathcal{H}=-\Delta+|x|^2$ is the harmonic oscillator on $\mathbb{R}^n$. Hence, for each $s\in\mathbb{R}^n$, we write
	\[\mathrm{Op}(\langle z\rangle^s) = (1+\mathcal{H})^{\frac{s}{2}}.\]
	
	For each $s\in\mathbb{R}$, we define the Shubin Sobolev space $H_\Gamma^s(\mathbb{R}^n)$ as
	\[H_\Gamma^s(\mathbb{R}^n)=H(\langle z\rangle^s)=\{u\in\mathcal{S}'(\mathbb{R}^n)\,:\, (1+\mathcal{H})^{\frac{s}{2}}u\in L^2(\mathbb{R}^n)\}.\]
	
	By the strong uncertainty principle, we have
	\[\mathcal{S}(\mathbb{R}^n)=\bigcap_{s\in\mathbb{R}}H^s_\Gamma(\mathbb{R}^n)\quad\text{and}\quad \mathcal{S}'(\mathbb{R}^n)=\bigcup_{s\in\mathbb{R}}H^s_\Gamma(\mathbb{R}^n).\]
	
	For a pseudo-differential operator $P\in OP\Gamma^m(\mathbb{R}^n)$, $m\geq 0$, condition \eqref{gh_loss} can be written as
	\begin{equation}\label{loss_shubin}
		u\in\mathcal{S}'(\mathbb{R}^n),\ p(x,D)u\in H_\Gamma^s(\mathbb{R}^n)\ \Rightarrow\ u\in H_\Gamma^{s+m'}(\mathbb{R}^n),
	\end{equation}
	for every $s\in\mathbb{R}$ and some $m'\leq m$. In this case, we have a loss of $m-m'$ in the order of the Shubin Sobolev spaces. Hence, Theorem \ref{thm_gh} has the following form:
	
	\begin{theorem}\label{thm_gh_shubin}
		Let $P\in OP\Gamma^m(\mathbb{R}^n)$, be a pseudo-differential operator satisfying \eqref{loss_shubin}, for every $s\in\mathbb{R}$ and a fixed $m'>0$. If $A\in OP\Gamma^{\tilde m}(\mathbb{R}^n)$, with $\tilde{m}<m'$, then $P+A$ satisfies \eqref{loss_shubin} for every $s\in\mathbb{R}$ and the same $m'$. In particular, $P+A$ is $\mathcal{S}$-globally hypoelliptic.
	\end{theorem}
	
	\begin{remark}
		Here, notice that the assumption $\tilde m<m'$ implies that
		\[\frac{\langle z\rangle^{m'}}{\langle z\rangle^{\tilde{m}}} = h(z)^{-(m'-\tilde{m})}.\]
		
		Then, we have $\varepsilon=m'-\tilde m>0$, so the ``$\varepsilon$-condition" on the statement of Theorem \ref{thm_gh} is satisfied.
	\end{remark}
	
	A symbol $p(z)\in\Gamma^m(\mathbb{R}^n)$, $m\in\mathbb{R}$, is $\Gamma$-elliptic if there exists $R>0$ such that
	\[|p(z)|\gtrsim \langle z\rangle^m,\quad |z|\geq R.\]
	
	More generally, given $0<\rho\leq 1$, a symbol $p(z)\in\Gamma^m(\mathbb{R}^n)$, $m\geq 0$, is $\Gamma_\rho$-hypoelliptic if there exists $0\leq m'\leq m$ and $R>0$ such that
	\[|p(z)|\gtrsim \langle z \rangle^{m'},\quad |z|\geq R\]
	and, for every $\gamma\in\mathbb{N}_0^{2n}$, we have
	\[|\partial_\gamma p(z)|\lesssim |p(z)|\langle z\rangle^{-\rho|\gamma|},\quad |z|\geq R.\]
	
	For $m<0$, the definition is the same with $m'\geq m$. If $p(x,\xi)$ is a $\Gamma_\rho$-hypoelliptic symbol, then $p(x,D)$ admits a parametrix. For every $s\in\mathbb{R}$ and $t<s+m'$, we have
	\[\|u\|_{H_\Gamma^{s+m'}}\lesssim \|p(x,D)u\|_{H_\Gamma^s}+\|u\|_{H_\Gamma^t},\]
	which implies \eqref{loss_shubin} and consequently the $\mathcal{S}$-global hypoellipticity of $p(x,D)$.
	
	Now, suppose that $p(x,D)$ is a differential operator. By \cite[Proposition 2.2.1]{NicRod}, the only differential operators in Shubin classes are those with polynomial coefficients, that is, operators of the form
	\[p(x,D) = \sum_{|\alpha|+|\beta|\leq m}c_{\alpha\beta}x^\beta D^\alpha,\quad c_{\alpha\beta}\in\mathbb{C},\]
	with symbol
	\[p(z) = \sum_{|\gamma|\leq m}c_\gamma z^\gamma,\quad c_\gamma\in\mathbb{C}.\]
	
	In this case, the $\Gamma_\rho$-hypoellipticity of $p(z)$ can be characterized simply by the estimate
	\[|\partial_z^\gamma p(z)|\lesssim |p(z)|\langle z \rangle^{\rho m},\quad |z|\geq R.\]
	
	Indeed, for some $\gamma\in\mathbb{N}_0^{2n}$ with $|\gamma|=m$, we obtain
	\[\langle z\rangle^{-\rho m}|p(z)|\gtrsim |\partial^\gamma p(z)| = C > 0,\quad |z|\geq R,\]
	which implies
	\[|p(z)|\gtrsim \langle z\rangle^{\rho m},\quad |z|\geq R.\]
	
	Finally, by Theorem \ref{thm_gh_shubin}, the $\mathcal{S}$-global hypoellipticity of differential operators in $OP\Gamma^m(\mathbb{R}^n)$ with $\Gamma_\rho$-hypoelliptic symbols is stable under perturbations by operators in $OP\Gamma^{\tilde m}(\mathbb{R}^n)$, with $\tilde m<\rho m$.
	
	\begin{example}[Schr\"odinger-type operators]
		Consider the operator $P=-\Delta+V(x)$ on $\mathbb{R}^n$, where		\[V(x)=\sum_{|\beta|\leq 2k}c_\beta x^\beta\]
		is a polynomial of degree $2k$, $k\in\mathbb{N}$. If the symbol 
		\[p(x,\xi)=|\xi|^2+V(x)\in \Gamma^2(\mathbb{R}^n)\]
		of $P$ is such that
		\[p(x,\xi)\neq 0,\quad\forall (x,\xi)\in\mathbb{R}^{2n}\setminus\{(0,0)\},\]
		then $p(x,\xi)$ is $\Gamma_{\rho}$-hypoelliptic, with $\rho=1/k$, see Proposition 2.5.3 and Example 2.5.5 in \cite{NicRod}. In this case, $P$ is $\mathcal{S}$-globally hypoelliptic and this property is stable under perturbation by pseudo-differential operators on $OP\Gamma^{m}$, with $m\leq 2/k$.
	\end{example}

	\subsection{SG classes}
	
	Now, if $m=(m_1,m_2)\in\mathbb{R}^2$, we define the SG classes by
	\[SG^m(\mathbb{R}^n) = S(\langle\xi\rangle^{m_1}\langle x\rangle^{m_2};\langle x\rangle,\langle\xi\rangle),\]
	where $\langle\xi\rangle^{m_1}\langle x\rangle^{m_2}$ is a regular weight, and denote by $OPG^{m}(\mathbb{R}^n)$ the set of operators with symbol in $SG^m(\mathbb{R}^n)$. It is clear that $\Phi(x,\xi)=\langle x\rangle$ and $\Psi(x,\xi)=\langle\xi\rangle$ satisfy the strong uncertainty principle.
	
	A symbol $p(x,\xi)\in SG^m(\mathbb{R}^n)$ is SG-elliptic if there exists $R>0$ such that
	\[|p(x,\xi)|\gtrsim \langle\xi\rangle^{m_1}\langle x\rangle^{m_2},\quad |x|+|\xi|\geq R.\]
	
	Given $m=(m_1,m_2),m'=(m_1',m_2')\in\mathbb{R}^2$, we denote $m'\leq m$ if $m_1'\leq m_1$ and $m_2'\leq m_2$. We define $m'<m$ the same way.
	
	Then, a symbol  $p(x,\xi)\in SG^m(\mathbb{R}^n)$, $(0,0)\leq m\in\mathbb{R}^2$, is SG-hypoelliptic if there exists $R>0$ and $m'\in\mathbb{R}^2$, $m'\leq m$, such that
	\[|p(x,\xi)|\gtrsim \langle\xi\rangle^{m_1'}\langle x\rangle^{m_2'},\quad |x|+|\xi|\geq R,\]
	and, for every $\alpha,\beta\in\mathbb{N}_0^n$,
	\[|\partial_\xi^\alpha\partial_x^\beta p(x,\xi)|\lesssim |p(x,\xi)|\langle\xi\rangle^{-|\alpha|}\langle x\rangle^{-|\beta|},\quad |x|+|\xi|\geq R.\]
	
	If $m<0$, the definition is the same, but with $m'\geq m$. Also, if we have $m_1\geq 0$ and $m_2<0$, we take $m'$ so that $m_1'\leq m_1$ and $m_2'\geq m_2$, the same if $m_1\leq 0$ and $m_2\geq 0$. For simplicity, we will restrict us to the case $m\geq 0$. In any case, an operator with SG-hypoelliptic symbol admits a parametrix.
	
	For each $s=(s_1,s_2)\in\mathbb{R}^2$, we define the weighted Sobolev space
	\[H^{s}(\mathbb{R}^n) = H(\langle\xi\rangle^{s_1}\langle x\rangle^{s_2}) = \{u\in\mathcal{S}'(\mathbb{R}^n)\,:\, \Lambda^s u\in L^2(\mathbb{R}^n)\},\]
	where $\Lambda^s=\mathrm{Op}(\langle\xi\rangle^{s_1}\langle x\rangle^{s_2})$. By the strong uncertainty principle, we have
	\[\mathcal{S}(\mathbb{R}^n)=\bigcap_{s\in\mathbb{R}^2}H^s(\mathbb{R}^n)\quad\text{and}\quad \mathcal{S}'(\mathbb{R}^n)=\bigcup_{s\in\mathbb{R}^2}H^s(\mathbb{R}^n).\]
	
	Then, in SG classes, \eqref{gh_loss} can be written as
	\begin{equation}\label{loss_SG}
		u\in\mathcal{S}'(\mathbb{R}^n),\ p(x,D)u\in H^s(\mathbb{R}^n)\ \Rightarrow\ u\in H^{s+m'}(\mathbb{R}^n),
	\end{equation}
	and Theorem \ref{thm_gh} has the following form:
	
	\begin{theorem}
		Let $P\in OPG^m(\mathbb{R}^n)$, $(0,0)\geq m\in\mathbb{R}^2$, be a pseudo-differential operator satisfying \eqref{loss_SG} for every $s\in\mathbb{R}^2$ and a fixed $m'\in\mathbb{R}^2$, $m'\leq m$. If $A\in OPG^{\tilde m}$, with $\tilde{m}<m'$, then $P+A$ satisfies \eqref{loss_SG} for every $s\in\mathbb{R}^2$ and the same $m'\in\mathbb{R}^2$. In particular, $P+A$ is $\mathcal{S}$-globally hypoelliptic.
	\end{theorem}
	
	\begin{remark} 
	As in the Shubin case, the ``$\varepsilon$-condition" of Theorem \ref{thm_gh} is satisfied by the assumption $\tilde m<m'$.
	\end{remark}
	
	If $p(x,D)$ is an operator with hypoelliptic symbol, the existence of a parametrix gives us that, for every $s=(s_1,s_2)\in\mathbb{R}^2$ and $t\in\mathbb{R}^2$, $t<s+m'$, we have the following estimate on weighted Sobolev spaces:
	\begin{equation}\label{est_sob_SG}
		\|u\|_{H^{s+m'}}\lesssim \|p(x,D)u\|_{H^s}+\|u\|_{H^t},
	\end{equation}
	which implies \eqref{loss_SG}.	
	
	As an example, consider the following class of operators introduced by I. Camperi in \cite{Camperi}:
	
	\begin{example}
		Given $\mu\in\mathbb{N}$ and $0<\gamma<1$, consider symbols of the form
		\[p(x,\xi)= \underset{0\leq|\beta|\leq\gamma\mu}{\sum_{0\leq|\alpha|\leq\mu}} c_{\alpha\beta}x^\beta\xi^\alpha,\]
		with $(x,\xi)\in\mathbb{R}^2$. Set
		\[\lambda(x,\xi) = 1+|\xi^\mu|+|x^{\gamma\mu}\xi^\mu|,\quad (x,\xi)\in\mathbb{R}^2.\]
		
		
		We say that $p(x,\xi)$ is $\lambda$-elliptic if there exists $R>0$ such that $p(x,\xi)\gtrsim \lambda(x,\xi)$ for $|x|+|\xi|\geq R$. By \cite[Proposition 1]{Camperi}, a $\lambda$-elliptic symbol is SG-hypoelliptic. Moreover, notice that
		\[\lambda(x,\xi)\gtrsim \langle \xi\rangle^\mu\langle x\rangle^{\gamma\mu}\]
		for large $|x|+|\xi|$. Hence, an operator $p(x,D):\mathcal{S}(\mathbb{R})\to \mathcal{S}(\mathbb{R})$ with $\lambda$-elliptic symbol $p(x,\xi)$ is $\mathcal{S}$-globally hypoelliptic and the $\mathcal{S}$-global hypoellipticity of $p(x,D)$ is stable under perturbations by operators in $OPG^{m'}(\mathbb{R})$, with $m'=(m_1',m_2')<(\mu,\gamma\mu)$.
	\end{example}
	
	
	The SG classes of symbols can be extended to SG manifolds, a class of non-compact manifolds introduced by E. Schrohe in \cite{Schr1987}. An SG manifold is defined as a smooth manifold $M$ equipped with an SG atlas, that is, a finite atlas whose coordinate changes behave like $SG^{(0,1)}(\mathbb{R}^n)$-type symbols. This invariance under such diffeomorphisms ensures that SG symbol classes and, consequently, SG operators are well-defined on these manifolds. Furthermore, the notion of SG-(hypo)elliptic symbols also remains invariant under these transformations.
	
	An important subclass of SG manifolds are manifolds with cylindrical ends. These can be regarded as compact manifolds from which a finite number of open discs have been removed, and along the boundary of each removed disc, an infinite cylinder is smoothly attached. For a precise definition, we direct the reader to Maniccia and Panarese \cite{MP2002}. More broadly, this class encompasses the interior of compact manifolds with boundary, endowed with a Riemannian metric that pushes the boundary to infinity. These are known as asymptotically Euclidean manifolds or scattering manifolds (see \cite{Melrose_GST}), and naturally include manifolds with cylindrical ends. For a detailed discussion, see the Appendix of Coriasco and Doll \cite{CD2021}.
	
	On these manifolds, the Schwartz space and a scale of weighted Sobolev spaces can be defined using the operator $\Lambda^{s_1,s_2}$. This operator has a local symbol$\langle \xi\rangle^{s_1}$ in bounded charts and $\langle\xi\rangle^{s_1}\langle x\rangle^{s_2}$ in unbounded charts. By working in local coordinates, a parametrix for operators with SG-hypoelliptic symbols can be constructed by gluing local parametrices with a suitable partition of unity, see \cite[Theorem 3.9]{Schr1987}. This yields estimates of the form \eqref{est_sob_SG} on the scale of weighted Sobolev spaces on the manifolds. Further details can be found in \cite{cordes}.

	\subsection{H\"ormander classes}
	
	When we set $M(x,\xi)=\langle\xi\rangle^m$, $\Phi(x,\xi)=\langle\xi\rangle^{\delta}$, and $\Psi(x,\xi)=\langle\xi\rangle^{\rho}$, for $m\in\mathbb{R}$, $0\leq\delta<\rho\leq 1$, our symbol classes reduce to the H\"ormander classes $S_{\rho,\delta}^m(\mathbb{R}^n)=S(M;\Phi,\Psi)$.
	
	An important observation is that the pair $\Phi,\Psi$ does not satisfy the strong uncertanty principle. Consequently, our main result, Theorem \ref{thm_gh}, does not hold for these classes when considering the usual scale of Sobolev spaces.
	
	Consider, for example, the Laplacian operator $\Delta$, with symbol $-|\xi|^2$. While $\Delta$ is elliptic in the usual sense and thus hypoelliptic in terms of singular supports, it is not $\mathcal{S}$-globally hypoelliptic. This is because its kernel contains the constant functions, that do not belong to $\mathcal{S}(\mathbb{R}^n)$.
	
	The usual Sobolev $H^s(\mathbb{R}^n)$ corresponds to $H^{s,0}(\mathbb{R}^n)$, the weighted Sobolev space of order $(s,0)$ as defined in the previous subsection.
	
	Let us examine the free particle operator $P_\lambda = -\Delta-\lambda$, with $\lambda\in\mathbb{C}$. As shown in \cite[Example 3.1.10]{NicRod}, $P_\lambda$ is SG-elliptic if and only if $\lambda\notin[0,+\infty)$.  Clearly, $P_\lambda$ is elliptic in $S^2_{1,0}(\mathbb{R}^n)$, and therefore satisfies estimates of the form
	\begin{equation}\label{loss_horm}
	\|u\|_{H^{s+2}}\lesssim \|P_\lambda u\|_{H^s}+\|u\|_{H^t}
	\end{equation}
	on the standard Sobolev spaces. Furthermore, since $P_\lambda$ is SG-elliptic, it is also $\mathcal{S}$-globally hypoelliptic. However, if we consider the 0-order perturbation $A=\lambda$, the operator $P_\lambda+A=-\Delta$ is no longer $\mathcal{S}$-globally hypoelliptic. This occurs despite the perturbation being of strictly lower order and $P_\lambda$ itself exhibiting no loss of derivatives.
	
	This example illustrates that elliptic estimates of type \eqref{loss_horm} on standard Sobolev spaces do not, in general, imply $\mathcal{S}$-global hypoellipticity, nor do they guarantee the stability of this property under certain perturbations. This fundamental difference arises because the strong uncertainty principle, which is central to our framework, does not hold in these symbol classes. Hence the intersection of their Sobolev spaces is not $\mathcal{S}(\mathbb{R}^n)$ and the hypoellipticity with loss of derivatives does not imply the $\mathcal{S}$-global hypoellipticity. Also, we have 
	\[\bigcap_{m\in\mathbb{R}} S^m_{\rho,\delta}(\mathbb{R}^n)\neq \mathcal{S}(\mathbb{R}^{2d}),\]	
	which is essential in order to construct the parametrix in Theorem \ref{parametrix}, see the proof of \cite[Theorem 1.3.6]{NicRod} for the details.

\bibliographystyle{plain}
\bibliography{references}
	
\end{document}